\documentclass[12pt]{amsart}
\usepackage[english]{babel}
\usepackage[pdftex]{graphicx}
\usepackage{wrapfig}
\usepackage{pdfsync}
\usepackage{epsfig}
\usepackage{epstopdf}
\usepackage{latexsym,amssymb,amsfonts,amsmath, amscd, euscript, MnSymbol}
\usepackage{enumerate}
\usepackage{verbatim}

\newcommand{\R}{\mathbb{R}}

\def\Power #1 { \powerset(#1) }
\def\Bidom #1 { {\mathfrak P} (#1) }

\newcommand{\Cov}{\mathop{\rm Cov}}
\newcommand{\Inc}{\mathop{\rm Inc}}
\newcommand{\inc}{\mathop{\rm Inc}}

\newcommand{\cf}{\mathop{\rm cf}}

\newtheorem{definition}{{\bf Definition}}[section]
\newtheorem{theorem}[definition]{{\bf Theorem}}

\newtheorem{corollary}[definition]{{\bf Corollary}}

\newtheorem{lemma}[definition]{\noindent {\bf Lemma}}

\newtheorem{claim}[definition]{\noindent {\bf Claim}}

\newtheorem{problem}[definition]{\noindent {\bf Problem}}

\newtheorem{fact}[definition]{\noindent {\bf Fact}}

\def\proofref #1 {{\noindent  {\bf Proof} (#1).}\ }

\def\endproof{\hfill {\kern 6pt\penalty 500
\raise -0pt\hbox{\vrule \vbox to5pt {\hrule width 5pt
\vfill\hrule}\vrule}}}
\def\centerpicture #1 by #2 (#3){\leavevmode
        \vbox to #2{
        \hrule width #1 height 0pt depth 0pt
        \vfill
        \special{pictfile #3}}}
\title[The  chain covering number]{The chain covering number of  a poset with no infinite antichains}
\email{uriabr@gmail.com}
\author[U.Abraham]{Uri Abraham} \address{Math \& CS Dept.,
Ben-Gurion University,
Beer-Sheva, 84105 Israel.}
\thanks{*This research benefited from the Abkasian seminar run throught the Covid epidemy and  was supported by  LABEX MILYON (ANR-10-LABX-0070) of Universit\'e de Lyon within the program ``Investissements d'Avenir (ANR-11-IDEX-0007)" operated by the French National Research Agency (ANR)}
\email{abraham@math.bgu.ac.il, uriabr@gmail.com}

\author[M.Pouzet]{Maurice Pouzet} \address{ICJ, Math\'ematiques, Universit\'e Claude-Bernard Lyon1, 43 bd. 11 Novembre 1918, 69622 Villeurbanne Cedex, France and Mathematics \& Statistics Department, University of Calgary, Calgary, Alberta, Canada T2N 1N4}

 \email{pouzet@univ-lyon1.fr }

\date{\today}


\begin{document}

\begin{abstract} The covering number $\Cov(P)$ of a poset $P$ is the least number of chains needed to cover $P$. For a cardinal $\nu$,  we give a   list  of posets  of cardinality and covering number $\nu$ such that  for every poset $P$ with no infinite antichain, $\Cov(P)\geq \nu$ if and only if $P$ embeds a member of the list. This list has two elements if $\nu$ is a successor cardinal, namely $[\nu]^2$ and its dual,  and four elements if $\nu$ is a limit cardinal with $\cf(\nu)$  weakly compact. For $\nu= \aleph_1$, a list was given by the first author;  his construction was extended by F. Dorais to every infinite successor cardinal $\nu$.  \end{abstract}

\maketitle

%
\section{Introduction}

Let $P:= (V, \leq)$ be an ordered set (poset). If $X$ is a subset of $V$, the poset \emph{induced}  by $P$ on $X$ is $P_{\restriction X}:= (X, \leq \cap X\times X)$.
 The \emph{chain covering number}  $\Cov(P)$ of $P$ is the least number of chains needed to cover $V$. 
 
 A famous result due to R.P. Dilworth \cite{dilworth} asserts that for every poset $P$, $\Cov(P)\geq n$, where $n$ is an integer,  if and only if $P$ contains an antichain of cardinality $n$. 
 
 If $n$ is replaced by $\aleph_0$, then $\Cov(P)\geq \aleph_{0}$ if and only if $P$ contains either an infinite antichain or the lexicographical sum of antichains of finite unbounded cardinality indexed by the chain of the integers or its reverse (see Lemma 4.1,  p. 491 of \cite{milner-pouzet2}).   
  
 It was conjectured by F. Galvin (1968) that   $\Cov(P)\geq \aleph_{1}$ if and only if $V$ contains a subset $X$ of cardinality $\aleph_1$ such that $\Cov (P_{\restriction X})= \aleph_1$. It is not known if the truth of this conjecture is consistent with ZFC. 
 The first author proved in \cite{abraham}  that this conjecture holds with the additional hypothesis that $P$ has no infinite antichain. He introduced \emph{Perles posets} having  cardinality  and chain covering number $\aleph_1$ and proved that if a poset $P$ has  no infinite antichain,  $\Cov(P)\geq \aleph_{1}$ if and only if some poset induced by $P$  is  isomorphic to a Perles poset. This result was extended by F. Dorais \cite{dorais2}    
 to infinite successors cardinal.

 It was conjectured by R. Rado \cite{rado} that Galvin's conjecture holds for interval orders (a poset $P$ is an \emph{interval order} if its vertices can be represented by intervals of a chain, a vertex $x$ being strictly before $y$ in the ordering if the interval associated to $x$ is to the left of the interval associated to $y$, these intervals being disjoint).  Todor\v{c}evi\'c \cite {todorcevic1} proved  that Rado's conjecture is consistent with the existence of a supercompact cardinal. In \cite{todorcevic2, todorcevic3}, Todor\v{c}evi\'c
proved that large cardinals are indeed necessary to establish the consistency of Rado’s
conjecture. He also proved that the truth of Rado's conjecture implies that Galvin's conjecture holds for $2$-dimensional posets (see \cite{dorais2}).

 In this paper, we prove Galvin's conjecture for posets with no infinite antichain and for each cardinal $\nu$ we give an explicit list  of posets of cardinality and chain covering number $\nu$. 
 
 In order to present our result, we need few notations. A  poset $P$ is \emph{embeddable} in a poset $P'$, and we set $P\leq P'$,  if $P$ is isomorphic  to the poset induced by $P'$ on some subset. Two posets $P$ and $P'$ are \emph{equimorphic} if each is isomorphic to an induced subposet of the other. The  \emph{dual} of $P$ is the poset on the same domain equipped with   the opposite order, we denote it by $P^*$.  Clearly, if $P\leq P'$ then $\Cov(P)\leq\Cov(P')$, while $\Cov(P)=\Cov(P^*)$. 

Let $\nu$ be an infinite ordinal. We denote by $\cf (\nu)$ the cofinality of $\nu$. Let $[\nu]^2: =\{(\alpha, \beta): \alpha<\beta<\nu\}$. We order  that set by setting $(\alpha, \beta)\leq (\alpha', \beta')$ if $\alpha\leq \alpha'$ and $\beta\leq \beta'$. As a subset of the direct product $\nu\times \nu$, this poset  is the upper half of  $\nu\times \nu$ ordered componentwise.

Our main result asserts: 
 \begin{theorem}\label{thm:covering}
  Let $\nu$ be an uncountable cardinal and $P$ be a poset with no infinite antichain.
 
 \begin{enumerate}
 
 \item  If $\nu$ is a successor cardinal, i.e., $\nu= \kappa^+$, then   $\Cov(P)\geq \nu$  if and only if $P$ contains either a copy of $[\nu]^2$ or of its dual. 
  
  \item If $\nu$ is a limit cardinal then $\Cov(P)\geq \nu$  if and only if $P$ contains  a poset $Q$, or its dual, of the form $\sum_{\alpha\in C}Q_{\alpha}$, where $C$ is a chain of cardinality $\cf (\nu)$, $Q_{\alpha}:= [\kappa_{\alpha}^+]^2$  and $(\kappa_{\alpha}^+)_{\alpha\in C}$ is a  family of distinct cardinals, all less  than $\nu$,  such that  $\bigvee \{\kappa_{\alpha}^+: \alpha\in C\} = \nu$.  \end{enumerate}
\end{theorem}

\noindent {\bf Comments about item 2.} Let $\lambda:= \cf(\nu)$. Suppose that  $\lambda$ is \emph{weakly compact}, that is every chain of cardinality $\lambda$ contains either a chain of order type $\lambda$ or its dual. Then, in item 2, we may replace the chain $C$ by $\lambda$ or its dual. If we compare the posets $Q$ by embeddability, we get only four nonequimorphic  posets, namely $\sum_{\alpha\in \lambda}[\kappa^+_{\alpha}]^2$, $\sum_{\alpha\in \lambda^*}[\kappa^+_{\alpha}]^2$, and their duals, two families  $(\kappa_{\alpha}^+)_{\alpha\in \lambda}$ giving equimorphic posets. 

If $\lambda$ is not weakly compact, then among the posets we obtain in item 2, some cannot be expressed as a well ordered or reversely well ordered chain  of type $\lambda$. 
For an example, suppose that $C$ is the chain $\R$ of real numbers.   Let $c$ be a bijective map from $\R$ onto $\kappa:= 2^{\aleph_0}$ and $f$ be a strictly increasing map from $\kappa$  in the collection of successor cardinals. If $A$ is any subchain of $C$ set $Q_{A}:= \sum _{a \in A}[f(c(a))]^2$. Then $\Cov(Q_{A})=\bigvee \{f(c(a)): a\in A\}$. Since a well ordered or reversely well ordered subchain $A$  of $\R$ is at most countable, and the cofinality of $2^{\aleph_0}$ is uncountable, $\Cov(Q_{A})<\Cov(Q_{C})$.

\begin{problem} Let  $\nu$ be  an uncountable limit cardinal  such that   $\lambda:=\cf(\nu)$ is not weakly compact, find the least number $G(\nu)$ of posets $Q$ needed in item 2?
\end{problem}
%


\section{Proof of Theorem \ref{thm:covering}}

\subsection {The easy part} 
The conditions in the theorem are sufficient to ensure that $\Cov(P)\geq \nu$. Indeed, 
\begin{enumerate} 

\item If $Q:=[\nu]^2$ then $\Cov (Q)= \nu$.

\item If $Q= \sum_{\alpha\in C}Q_{\alpha}$, where $C$ is a chain of cardinality $\cf (\nu)$, $Q_{\alpha}:= [\kappa_{\alpha}^+]^2$  and $(\kappa_{\alpha}^+)_{\alpha\in C}$ is a  family of distinct cardinals, all less  than $\nu$ such that  $\bigvee \{\kappa_{\alpha}^+: \alpha\in C\} = \nu$,  then $\Cov(Q)= \nu$.

\end{enumerate}

\begin{proof}
We prove Item 1 under the requirement that $\nu$ is a successor cardinal. Suppose that $Cov(Q):= \mu<\nu$. Let $(L_{\gamma})_{\gamma<\mu}$ be a covering of $Q$. It induces a covering of $Q_{\alpha}:= \{(\alpha, \beta): \alpha<\beta< \nu\}$ for each  $\alpha <\nu$. Since $\nu$ is a regular cardinal, there is some $\alpha_{\gamma}<\nu$ such that $L_{\alpha_{\gamma}}\cap Q_{\alpha}$ has cardinality $\nu$. If  $f(\alpha)$ denotes  the least one then the map $f$ is one to one (otherwise, if $f(\alpha)=f(\alpha')$ for some $\alpha \not= \alpha'$, $L_{f(\alpha)}\cap Q_{\alpha}$ and $L_{f(\alpha')}\cap Q_{\alpha'}$ would be cofinal in each other), hence $\nu\leq \mu$. A contradiction. If $\nu$ is a limit,  we observe that $Q$ contains a sum as in Item 2. Finally, we note that  Item 2 follows from Item 1 and the fact that $\Cov(Q)\geq \bigvee_{\alpha \in C}  \Cov (Q_{\alpha})$. 
\end{proof}

\subsection{The substantial part}
We prove a reduction lemma (Lemma \ref{lem:covering2}) leading to a  a reduction theorem (Theorem \ref{thm:covering2})  from which Theorem \ref{thm:covering} follows easily.

We  recall that if $P:= (V, \leq )$ is a poset, two elements $x, y$ of $V$ are \emph{comparable} if either  $x\leq y$ or $y\leq x$; otherwise there are \emph{incomparable}.  The \emph{incomparability graph} of $P$ is the graph $\Inc (P)$ vith vertex set $V$ and edges the pairs $\{x, y\}$ of incomparable elements.   For each subset $A$ of $P$,   we  set $\inc_A(P)$ for the set of vertices  of $P$ which are incomparable to every member of $A$. We set $\downarrow A:= \{x\in P: x\leq a\;   \text {for some } a\in A\}$; we define similarly $\uparrow A$. If $A= \{a\}$, we denote these sets by $\inc_a(P)$, $\downarrow a$ and $\uparrow a$. If $a\leq b$ in $P $ we set  $[a,b]:=  \{ z\in P: a\leq z\leq b\}$ instead of $\uparrow a\cap \downarrow b$.

If $G$ is an undirected graph, two vertices $x$ and $y$ of $G$ are \emph{connected} if there is some finite path joining $x$ to $y$. The blocks of this equivalence relation are the \emph{connected components} of $G$. The graph 
$G$ induces a graph  $G/{\equiv}$ on the set of connected components, this graph  has no edges, hence $G$ is the direct  sum of its connected components  indexed by  $G/{\equiv}$. 

If $G$ is the incomparability graph $\Inc(P)$ of a poset $P$ then  the graph $G/{\equiv}$ is the incomparability graph of a chain.  This yields the following  result belonging  to the folklore of the theory of ordered sets. 

\begin{lemma}  Let $P$ be a poset then  $P$ is the lexicographical  sum of the posets $P_i:= P_{\restriction C_i}$ where each $C_i$ is a connected component of $\Inc(P)$, indexed by a chain $D$;  that is $P:= \sum_{i\in D} P_i$.  
\end{lemma}

From this we deduce easily the following two results. 
\begin{corollary}
If  $P$ is a poset and $(P_i)_{i\in D}$ is the set of the posets induced on the   connected component of $P$ then:
\begin{equation}\label {eqsum}
\text{Either}\;  \Cov (P)= \Cov(P_i) \;  \text{for some}\;   i\in D,  \; \text{or}\;  \Cov (P)= \bigvee_{i\in D}  \Cov(P_i).
\end{equation}
\end{corollary}

\begin{lemma}\label{lem:covering3}
  Let $\nu$ be an uncountable cardinal and $P$ be a poset with no infinite antichain such that $\Cov (P)\geq \nu$ and $\Cov (P_{\restriction C})<\nu$ for each   
 connected component $C$ of $\Inc (P)$. Then $\nu$ is a limit cardinal and   if $(\kappa^+_{\alpha})_{\alpha<\cf (\nu)}$ is a cofinal sequence in $\nu$, there is an injection  $\alpha \rightarrow C_{\alpha}$ from $\cf(\nu)$ in the set $D$ of connected components of $\Inc(P)$ such that for each $\alpha <\cf(\nu)$,  $\Cov (P_{\restriction C_{\alpha}})\geq \kappa^+_{\alpha}$. 
 \end{lemma}

\begin{lemma}\label{lem:metric} Let $P$ be a poset and $x, y$ be two  distinct vertices in the same connected component of $\Inc(P)$. We set $d_P(x,y)=0$ if $x=y$ and otherwise $d_P(x,y):= n+1$ where  $n$ is the least integer such that there is a path $x=:x_0, \dots, x_{n+1}:= y$ in $\Inc(P)$ joining $x$ and $y$. Suppose that $x<y$.  Then:
\begin{enumerate}
\item $x_i<x_j$ for $i+2\leq j$;
\item $[x,y] \subseteq \Inc_{x_1}(P) \cup \dots \cup \Inc_{x_n}(P)$.
\end{enumerate} 
 
\end{lemma} 
\begin{proof} Item 1 is Lemma I.2.2 p.5 of \cite{pouzet78}.
Item 2. Induction on $n$. For $n:= 1$, we have trivially $[x_0, x_2]  \subseteq \Inc_{x_1}(P)$. Let $n>1$. Let  $z\in  [x,y]$. If $z\leq x_n$, we may apply induction to $[x, x_n]$. Then    $z\in \Inc_{x_1}(P) \cup \dots \cup \Inc_{x_{n-1}}(P)$. If $z\not\leq x_n$, then $z \in \Inc_  {x_n}(P)$ (otherwise $x_n<z$ and since $z<x_{n+1}$, $x_n<x_{n+1}$ which is impossible. 
\end{proof}


\begin{lemma}\label{lem:covering2}
  Let $\nu$ be an uncountable cardinal and $P$ be a poset with no infinite antichain such that $\Cov (P)\geq \nu$. Then $P$ contains an induced subposet $Q$ such that $\Cov (Q)\geq \nu$ and 
  \begin{enumerate} \item  either $Q$ or its dual satisfies  $\Cov(Q\setminus \uparrow x)< \nu$ for every $x\in Q$, or  \item 
 each connected component $Q_i$ of $Q$ satisfies $\Cov(Q_i)<\nu$.
 \end{enumerate} 
 \end{lemma}
The proof relies on the next  two claims. 
\begin{claim}\label{claim1} $P$ contains an induced subposet $Q$ such that $\Cov(Q)\geq \nu$ and  
$\Cov(Q_{\restriction 
\inc_Q (x)})< \nu$ for every $x\in Q$.
 \end{claim} 
 
\noindent{ \bf Proof of Claim \ref{claim1}.} 
If there is no $x\in P$ such that $\Cov (P_{\restriction \inc_x(P)})\geq \nu$, set $Q:=P$. If not consider the collection of antichains $A$ of $P$ such that $\Cov (P_{\restriction \inc_A(P)})\geq \nu$. The  antichains of $P$ being finite, one of these antichains,  say $L$, is maximal w.r.t. inclusion. The poset $Q:= P_{\restriction \inc{L}(P)}$ satisfies the conclusion of the lemma. 
\hfill $\Box$

\begin{claim}\label{claim2} Let $P$ be a poset with no infinite antichain such that $\Cov(P)\geq \nu$ and  
$\Cov(P_{\restriction \inc_x (P)})< \nu$ for every $x\in P$. Then, \begin{enumerate}[{(a)}]
\item either $\Cov(P_{\restriction \uparrow x})\geq  \nu$ or $\Cov(P_{\restriction \downarrow x})\geq  \nu$ for every $x\in P$.
\item Let  $x_0\in P$ such that $Q:=P_{\restriction \uparrow x_0}$ satisfies  $\Cov(Q)\geq  \nu$. If   $\inc(P)$ is connected,    then  $\Cov(Q \setminus  \uparrow y)< \nu$ for every $y\in Q$.
 \end{enumerate}
 \end{claim} 
 \noindent{ \bf Proof of Claim \ref{claim2}.} Item (a).  Let $x\in P$. We have $P= \downarrow x\cup \uparrow x \cup \inc_x(P)$. Hence, $\Cov(P)\leq \Cov  (P_{\restriction \downarrow x}) \cup  \Cov (P_{\restriction \uparrow x}) \cup \Cov(P_{\restriction  \inc_x(P)})$. Since $\Cov(P_{\restriction \inc_x(P)})< \nu$, either $\Cov  (P_{\restriction \downarrow x})\geq \nu$ or $\Cov (P_{\restriction \uparrow x})\geq \nu$. \hfill $\Box$
  
  Item (b) Let $y\in Q$. We may suppose $y\not =x_0$, otherwise since $Q\setminus \uparrow y=\emptyset$, $\Cov( Q\setminus \uparrow  y)<  \nu$ as claimed. Let $n$ be the least integer such that there is a path $x_0, \dots, x_{n+1}= y$ in $\Inc(P)$ joining $x_0$ and $y$.  According to Item 2 of Lemma \ref{lem:metric},  $[x_0, x_{n+1}] \subseteq \inc_{x_1}(P) \cup \dots \cup \inc_{x_{n}}(P)$. Since $\Cov(P_{\restriction \inc_x(P)})< \nu$ for each $x\in P$, we have $\Cov(Q_{\restriction [x_0,  x_{n+1}]})<\nu$. Since  $Q\setminus \uparrow y \subseteq [x_0, y] \cup \inc_y(P)$,  $\Cov(Q\setminus \uparrow y)<\nu$ as claimed.  \hfill $\Box$ 
 
 With these two claims, the proof of Lemma \ref{lem:covering2} goes as follows. We start with $P$, we select $Q$ given by  Claim \ref{claim1}. We decompose $Q$ in connected components $Q_i$'s. If $\Cov(Q_i)\geq \nu$ for some 
 $i$, we apply Claim \ref{claim2}. We pick $x_0\in Q_i$. By item (a), we may suppose  that  $\Cov(Q_i \restriction{\uparrow x_0})\geq \nu$ for some $x_0\in Q_i$. Then we apply Item (b) and $(1)$ holds. Otherwise  $\Cov(Q_i)< \nu$ for each $i$ and $(2)$ holds.
 \hfill $\Box$

  This reduction lemma allows to deduce easily the proof of Theorem \ref{thm:covering} from the following result.
 
  \begin{theorem}\label{thm:covering2}
  Let $\nu$ be an uncountable cardinal and $P$ be a poset with no infinite antichain such that $\Cov (P)\geq \nu$ and $\Cov(P\setminus \uparrow x)< \nu$ for every $x\in P$. 
 \begin{enumerate}
 
 \item  If $\nu$ is a successor cardinal, i.e., $\nu= \kappa^+$, then   $P$ contains  a copy of $[\nu]^2$. 
  
  \item If $\nu$ is a limit cardinal $\nu>\omega$ then $P$ contains a poset $Q$ of the form  $\sum_{\alpha<\cf (\nu)} Q_{\alpha}$ or of the form $\sum_{\alpha <\cf (\nu)}Q_{\alpha}^{*}$, where $Q_{\alpha}:=[\kappa^+_{\alpha}]^2$, $(\kappa_{\alpha}^+)_{\alpha< \cf(\nu)}$ is a   family of distinct cardinals, all less  than $\nu$ and $\bigvee \{\kappa_{\alpha}^+: \alpha< \cf (\nu)\} = \nu$.  
  
  \end{enumerate}
\end{theorem}

The proof of Theorem \ref{thm:covering2} will occupy the next section.

We conclude this section by a deduction of Theorem \ref{thm:covering} from Theorem \ref{thm:covering2}.
\subsection {Proof of Theorem \ref{thm:covering}}
We argue by induction on the cardinality $\nu$. 
Let $P$ be a poset with no infinite antichain such that $\Cov(P)\geq \nu$. 
Let $Q$ be given by  Lemma \ref{lem:covering2}.  If $\Cov (Q_{\restriction C} )<\nu$ for each connected component $C$ of $\Inc(Q)$, then $Q$ contains a sum $\sum_{C_{\alpha}}  P_{\restriction C_{\alpha}}$ where $\Cov(P_{\restriction C_{\alpha}} )\geq \kappa^+_{\alpha}$. Via the induction hypothesis, Item 1 of Theorem \ref{thm:covering} holds for uncountable successor cardinals smaller than $\nu$, hence   $P_{\restriction C_{\alpha}}$ embeds either $[\kappa^+_{\alpha}]^2$ or its dual. It follows that Item 2 holds. If  $Q$ or its dual satisfies  $\Cov(
 Q\setminus \uparrow x)< \nu$ for every $x\in Q$, we may suppose that $\Cov(
 Q\setminus \uparrow x)< \nu$ for every $x\in Q$. We apply then Theorem \ref{thm:covering2}. 

\section{Proof of Theorem \ref {thm:covering2}}
 We suppose that the result holds for uncountable cardinals smaller than $\nu$. 
\subsection{Proof of $(2)$ of Theorem \ref {thm:covering2}.} This  is the simpler one. Let $\theta:= \cf (\nu)$. 

\begin{fact}\label{fact1} Every subset $X\subseteq P$ with $\vert X \vert<\theta$  is bounded above in $P$.
\end{fact} 
\noindent {\bf Proof of Fact \ref{fact1}.}
Otherwise $P= \bigcup_{x\in X} (P\setminus \uparrow x)$. Since by hypothesis $\Cov (P\setminus \uparrow x)<\nu$, we must have $\vert X\vert \geq \cf(\nu)= \theta$, a contradiction. \hfill $\Box$
 
 A subset $A$ of $P$ is \emph{cofinal} in $P$ if every element of $P$ is majorized by some element $a\in A$, i.e, $\downarrow A=P$. The \emph{cofinality} of $P$ is the least cardinality of a cofinal subset of $P$, we denote it by $cf (P)$.  
 
  \begin{fact}\label{fact2}  $cf(P)\geq \theta$. 
 \end{fact} 
 \noindent {\bf Proof of Fact \ref{fact2}.}
Since $\Cov (P)\geq \nu$  and $\Cov(P\setminus \uparrow x)< \nu$ for every $x\in P$, $P$ has no maximal element. In particular it has no largest element. Thus a sequence $(x_{\alpha})_{\alpha<\cf(P)}$  of elements of $P$ which is cofinal in $P$ cannot be bounded above in $P$, so from Fact 1 its cardinality  is at least $\theta$.  \hfill $\Box$

Let $(\kappa_{\alpha}^+)_{\alpha< \theta}$ be a sequence  of cardinal numbers, all less  than $\nu$,  such that  $\bigvee \{\kappa_{\alpha}^+: \alpha<\theta\} = \nu$.  

\begin{fact}\label {fact3} 
There is an increasing  sequence $(x_{\alpha})_{\alpha< \theta}$ of elements of $P$ such that $\Cov ([x_{\alpha}, x_{\alpha+1} [) \geq \kappa^{+}_{\alpha}$  for every $\alpha <\theta$. 
\end{fact} 
\noindent {\bf Proof of Fact \ref{fact3}.} Pick $x_0$ arbitrary in $P$. Since $\Cov(P)\geq \nu$ and $\Cov (P \setminus \uparrow x_0)  \leq \nu$, $\Cov (\uparrow x) \geq \nu > \kappa_{0}^{++}$. Via the induction hypothesis on $\nu$, $\uparrow x$ contains  a copy of $[\kappa_{\alpha}^{++}]^2$ or its reverse, hence there is some element $x_1\in P$ such that $\Cov ([x_0, x_1[)\geq \kappa^{+}_0$. Suppose $(x_{\alpha})_{\alpha< \mu}$ be defined up to some $\mu<\theta$. Apply Fact \ref{fact1}, get an element $x_{\mu}$ majorizing all the $x_{\alpha}'s$. \hfill $\Box$. 

Applying the induction hypothesis, we get that each interval $[x_{\alpha}, x_{\alpha+1}[$ contains a copy of $[\kappa^{+}_{\alpha}]^2$ or of his dual. Statement $(2)$ follows. 

\subsection{Proof of $(1)$ of Theorem \ref {thm:covering2}.} 
The proof relies on the Erd\"os-Dushnik-Miller's partition theorem, on  a   description of cofinal subsets of posets with no infinite antichains due to the second author \cite{pouzet} (see  \cite{milner-prikry, milner-pouzet1, fraisse})  and  on the notion of  purity and Theorem 24 of \cite{assous-pouzet}.

\begin{theorem}\label{thm:EDM}  Erd\"os-Dushnik-Miller's partition theorem \cite{dushnik-miller} $$\lambda\rightarrow (\aleph_0, \lambda)^2$$ for every infinite cardinal $\lambda$. That is,  every graph with $\lambda$ vertices contains either an  infinite independent set or   a complete subgraph on $\lambda$ vertices. 
\end{theorem} 
\begin{theorem}\label{cofinalthm} Every poset with no infinite antichain has a cofinal subset which is isomorphic to a finite direct sum $A_1 \bigoplus\dots \bigoplus A_k$ where each $A_i$ is a finite product $\beta_{1,i}\times \dots \times \beta_{n_i, i}$  of regular distinct ordinals. 
\end{theorem} 
\begin{definition}\label{def:pure} A  poset $P$ is \emph{pure} if every proper initial segment $I$ of $P$ is strictly  bounded above (that is some $x\in P\setminus I$ dominates $I$). 
 \end{definition}
 
 This condition amounts to the fact that every non cofinal subset of $P$ is strictly bounded above (indeed, if a subset $A$ of $P$ is not cofinal, then $\downarrow A\not = P$ hence from purity, $\downarrow A$, and thus $A$, is strictly bounded above. The converse is immediate). 

%
 We recall
 
 \begin{theorem}\label{theo:pure}
Let $P$ be a  poset with  infinite  uncountable cofinality $\nu$. Then $P$ is pure if and only if  $P$ is a lexicographical  sum $\sum_{a\in K }P_{a}$ where $K$  is  a chain of order type $\nu$. \end{theorem}
Our first lemma is the following:
\begin{lemma}\label{lem:cofinalchain} Let $\nu:= \kappa^+$ be an infinite  successor cardinal and $P$ be a poset with no infinite antichain. If $\Cov(P) \geq \nu$ and $\Cov(P\setminus \uparrow x)<\nu$ for every $x\in P$ then there is a chain $C$ cofinal in $P$  with well ordered  order type $\lambda\geq \nu$ and for every $x\in P$, $\uparrow x$ is impure.  
\end{lemma}

\begin{proof}Pick a cofinal subset as in Theorem \ref{cofinalthm}. Since $\Cov (P)\geq \nu$ then $\Cov (\downarrow A_i)\geq \nu$ for some $i$, with $1\leq i\leq k$. But,  since  $\Cov (\uparrow x)<\nu$ for   every $x\in P$, $\downarrow A_i= P$. This implies $k=1$.  We prove next that $A_1$ is a chain.
\begin{claim}\label {claimordinal} 
 Each ordinal $\beta_{j,1}$, for $1\leq j\leq n_1$ is at least $\nu$. 
 \end{claim}
 \noindent{\bf Proof of Claim \ref{claimordinal}.} 
 Suppose that $\beta_{j,1}<\nu$. Since $\nu$ is a successor cardinal and $\beta_{j,1}$ is a  cardinal, $\beta_{j,1} \leq \kappa$, the predecessor of $\nu$. Let $X_j:= \{x_{\alpha}: \alpha<\beta_{j,1}\}$ be the image in $P$ of the chain made of the elements of $A_1$ of the form $(z_{k})_{k=1, \dots, n_1}$ where $z_{k}=0$ for $k\not =j$ and $z_k<\beta_{j,1}$ otherwise. Since $\Cov( P\setminus \uparrow x_{\alpha}) \leq \kappa$ and $\beta_{j, 1}\leq \kappa$,  
 $\Cov( \bigcup_{\alpha<\beta_{j,1}} (P\setminus \uparrow x_{\alpha})) \leq \kappa$.  This union is $P\setminus \bigcap_{\alpha<\beta_{j,1}}\uparrow x_{\alpha}$. That is $P$ since in $P$ the chain $X_j$ is unbounded. Contradicting $\Cov (P) \geq \kappa^+$. \hfill $\Box$

 Since in $A_1$ all ordinals $\beta_{j, i}$ are distincts, if there are two, then $\nu \times \nu^+$ embeds in $A_1$.  For $u:= (0, \nu)\in Q:= \nu\times \nu^+$,   $\Cov (Q\setminus \uparrow u)= \nu$ since that set contains $\nu\times \nu$ hence $\Cov (P\setminus \uparrow x)\geq \nu$ for some $x\in P$ contradicting our hypothesis. Consequently, there is just one ordinal, namely $\beta_{1,1}$, hence $P$ has a cofinal chain of type at least $\nu$.

In order to complete the proof of the lemma, suppose by way of contradiction that  $\uparrow x$ is pure for some $x\in P$. According to Theorem \ref{theo:pure}, the final segment  $\uparrow x$ is   a lexicographical  sum $\sum_{a\in K }P_{a}$ where $K$  is  a chain with  cofinality $cf(P)$.
For each $a\in K$, 
 $\Cov(P_{a}) \leq \kappa$, from which follows that  $\Cov(\uparrow x) \leq \kappa$. Since $\Cov (P\setminus \uparrow x)\leq \kappa$, we have $\Cov(P) \leq \kappa$. This contradicts our hypothesis that $\Cov(P)= \nu=\kappa^+$. 
\end{proof} 

\begin{lemma}\label{lem:secondcofinalchain} Let $\lambda$ be an uncountable regular cardinal, let $P$ be a poset with no infinite antichain, having  a cofinal chain $C$ of type $\lambda$, and such that $\uparrow x$ is impure for each $x\in P$. Then there is a chain (for set inclusion)  $\mathfrak I$ of type $\lambda$ made of unbounded ideals of $P$, each one containing a cofinal chain of type $\lambda$. 
\end{lemma}
\begin{proof} Let $\mathfrak {I}_P(\lambda)$ be the set of unbounded ideals of $P$ containing a cofinal chain of type $\lambda$. By hypothesis, $P\in \mathfrak {I}_P(\lambda)$. Let $\mathfrak{I}_P(\lambda)^{-}:= \mathfrak{I}_P(\lambda)\setminus \{P\}$. We order $\mathfrak{I}_P(\lambda)^{-}$ by set inclusion. 

\begin{claim}\label{claim:cof}
$\cf(\mathfrak{I}_P(\lambda)^{-})\geq \lambda$. 
\end{claim} 
\noindent{ \bf Proof of Claim \ref{claim:cof}.} Let $(J_{\alpha})_{\alpha<\mu}$ with $\mu:=\cf(\mathfrak{I}_P(\lambda)^{-})$ be a  sequence cofinal in $\mathfrak{I}_P(\lambda)^{-}$.For each $\alpha<\mu$, pick $x_{\alpha} \in C\setminus J_{\alpha}$. If $\mu<\lambda$ then,  since $\lambda$ is a  regular cardinal, there is some $x\in C$ majorizing all the $x_{\alpha}'s$. In order to get a contradiction, it suffices to show that there is some $J$ such that $x\in J\in \mathfrak{I}_P(\lambda)^{-}$. According tyo our hypothesis, the final segment $\uparrow x$ is impure. Hence, it contains a proper initial  segment $I$ which is unbounded (in $\uparrow x$). Let $C_x:= \uparrow x\cap C$. For each $y\in C_x$, we may pick some $ c_y\in I\setminus \downarrow y$ since $y$ does not majorize $I$.  Apply Erd\"os-Dushnik-Miller's Theorem (Theorem \ref{thm:EDM}) to the sequence $(c_y)_{y\in C_x}$. Since $P$ has no infinite  antichain, there is an increasing subsequence  
 $(c_y)_{y\in D}$ with $D$ cofinal in $C_x$. The values of this sequence cannot be bounded in $P$. (otherwise, if $a$ is an upper bound, pick $y\geq a$ in $D$ and get
a  contradiction with $c_y$) hence they generate an unbounded ideal; since it is include in $\downarrow I$ it is distinct of $P$. \hfill $\Box$

\begin{claim}\label{claim:infinite antichain}
$\mathfrak{I}_P(\lambda)^{-}$ has no infinite antichain. 
\end{claim} 
\noindent{ \bf Proof of Claim \ref{claim:infinite antichain}.} This relies on the fact that $\lambda$ is uncountable and regular. Indeed, suppose by way of contradiction that $\mathfrak{I}_P(\lambda)^{-}$ contains an infinite antichain $(J_n)_{n<\omega}$. For each $n\not =m<\omega$ there is some $x_{n,m}\in I_n \setminus J_m$. Since $I_n$ contains a cofinal chain of regular type distinct from $\omega$, it contains an element $x_n$ majorizing all the $x_{n, m}$. Then $\{x_n: n<\omega\}$ is an infinite antichain in $P$. Impossible. \hfill $\Box$ 

 With that in hands, we complete the proof of Lemma \ref {lem:secondcofinalchain} as follows. We pick a well founded cofinal subset in 
 $\mathfrak{I}_P(\lambda)^{-}$. It is well quasi-ordered by Claim \ref{claim:infinite antichain} and, by  Claim \ref{claim:cof},  has cardinality at least $\lambda$. By Erd\"os-Dushnik-Miller's  theorem, or more simply Konig's  lemma, it contains a chain of type $\lambda$. 
\end{proof}

\begin{lemma}\label{lem:thirdcofinalchain} Let $\lambda$ be a regular cardinal and $P$ be a poset. If $P$ is the union of a chain 
$(J_{\alpha})_{\alpha<\lambda}$ of unbounded ideals,  each containing a cofinal chain of type $\lambda$, then $P$ contains a cofinal subset isomorphic to $[\lambda]^2$.
\end{lemma}
\begin{proof} We may suppose that $\check J_{\alpha} :=  J_{\alpha}\setminus \bigcup_{\alpha'<\alpha} J_{\alpha'}\not = \emptyset$ for everty $\alpha<\lambda$. We   define an embedding  from $[\lambda]^2$ in $P$ such that $f(\alpha, \beta)\in \check J_{\alpha}$ for every $\alpha<\beta<\lambda$. Pick $f(0,1)$ arbitrary in $J_0$. Next,  order $[\lambda]^2$ lexicographically according to the second difference, that is $(\alpha' \beta')\prec (\alpha, \beta)$ if either $\beta' <\beta$ or $\beta= \beta'$ and $\alpha'<\alpha$. Let $(\alpha, \beta)$. Suppose $f$ be defined for all $(\alpha', \beta')\prec (\alpha, \beta)$. We extend it to $(\alpha, \beta)$. For that, suffices to show that $\check J_{\alpha}$ contains some $x_{\alpha, \beta}$ larger that   $f(\alpha, \beta' )$ for all $\alpha <\beta'<\beta$ and larger than  $f(\alpha', \beta )$ for all $\alpha'<\beta$.  \\Case 1.  $\alpha=0$. Since $J_0$ contains  a cofinal chain, one may find $x_{0, \beta}\in J_0$  with $f(0, \beta') < x_{0, \beta}$ and $x_{0, \beta}\not \leq f (\alpha', \beta')$ for $\alpha'<\beta'\leq \beta$.\\ Case 2. $\alpha>0$. One may find $x_{\alpha, \beta}\in \check J_{\alpha}$  with $f(\alpha', \beta') < x_{\alpha , \beta}$ for $\alpha'<\beta'\leq \beta$   and $x_{\alpha, \beta}\not \leq f (\alpha', \beta')$ for $\alpha'<\beta'\leq \beta$.  One check that the map $f$ is an embedding. 
\end{proof} 

To deduce the proof of $(1)$ we proceed as follows. Let $P$ be as in Theorem \ref{thm:covering2} and $\nu:= \kappa^+$. Apply successively Lemmas \ref {lem:cofinalchain}, \ref{lem:secondcofinalchain} and \ref{lem:thirdcofinalchain}, get a copy of $[\lambda]^2$ where $\lambda:= \cf(P)$. Since $\lambda\geq \nu$, this yields  the required conclusion.

\end{document}